\documentclass[oneside,english]{amsart}
\usepackage[T1]{fontenc}
\usepackage[latin9]{inputenc}
\usepackage{float}
\usepackage{amsthm}
\usepackage{amssymb}
\usepackage{graphicx}
\usepackage{esint}

\makeatletter
\numberwithin{equation}{section}
\numberwithin{figure}{section}
\theoremstyle{plain}
\newtheorem{thm}{\protect\theoremname}
  \theoremstyle{plain}
  \newtheorem{cor}[thm]{\protect\corollaryname}
  \theoremstyle{plain}
  \newtheorem{prop}[thm]{\protect\propositionname}

\makeatother

\usepackage{babel}
  \providecommand{\corollaryname}{Corollary}
  \providecommand{\propositionname}{Proposition}
\providecommand{\theoremname}{Theorem}

\begin{document}

\title{Approximating $\overline{z}$ in the Bergman Space}

\author{Matthew Fleeman and Dmitry Khavinson}
\begin{abstract}
We consider the problem of finding the best approximation to $\bar{z}$
in the Bergman Space $A^{2}(\Omega)$. We show that this best approximation
is the derivative of the solution to the Dirichlet problem on $\partial\Omega$
with data $\left|z\right|^{2}$ and give examples of domains where
the best approximation is a polynomial, or a rational function. Finally,
we obtain the ``isoperimetric sandwich'' for $dist(\overline{z},\Omega)$
that yields the celebrated St. Venant inequality for torsional rigidity. 
\end{abstract}
\maketitle

\section{Introduction}

Let $\Omega$ be a bounded domain in $\mathbb{C}$ with boundary $\Gamma$.
Recall that the Bergman space $A^{2}(\Omega)$ is defined by:
\[
A^{2}(\Omega):=\{f\in Hol(\Omega):\left\Vert f\right\Vert _{A^{2}(\Omega)}^{2}=\int_{\Omega}\vert f(z)\vert^{2}dA(z)<\infty\}.
\]
In \cite{GuadKhav} the authors studied the question of ``how far''
$\bar{z}$ is from $A^{2}(\Omega)$ in the $L^{2}(\Omega)$-norm.
They showed that the best approximation to $\bar{z}$ in this setting
is 0 if and only if $\Omega$ is a disk, and that the best approximation
is $\frac{c}{z}$ if and only if $\Omega$ is an annulus centered
at the origin. In this note, we examine the question of what the best
approximation looks like in other domains. In section 2, we characterize
the best approximation to $\bar{z}$ as the derivative of the solution
to the Dirichlet problem on $\Gamma$ with data $\left|z\right|^{2}.$
This shows an interesting connection between the Dirichlet problem
and the Bergman projection. Recently in \cite{Legg}, A. Legg noted
independently another such connection via the Khavinson-Shapiro conjecture.
(Recall that the latter conjecture states that ellipsoids are the
only domains where the solution to the Dirichlet problem with polynomial
data is always a polynomial, cf. \cite{LundRend} and \cite{Render}.
In \cite[Proposition 2.1]{Legg}, the author showed that in the plane
this happens if and only if the Bergman projection maps polynomials
to polynomials). In section 3 we look at specific examples. In particular
we look at domains for which the best approximation is a monomial
$Cz^{k}$, some examples where the best approximation is a rational
function with simple poles, as well as examples where the best approximation
is a rational function with non-simple poles. In section 4, we prove
two isoperimetric inequalities, and obtain the St. Venant inequality.

\par\vspace{0.1in}

\thanks{\emph{Acknowledgment: }The final draft of this paper was produced
at the 2015 conference \textquotedbl{}Completeness problems, Carleson
measures and spaces of analytic functions\textquotedbl{} at the Mittag-Leffler
institute. The authors gratefully acknowledge the support and the
congenial atmosphere at the Mittag-Leffler institute. We would also
like to thank Jan-Fredrick Olsen for kindly pointing out to us his
results that led to the proof of Theorem \ref{thm:Iso-Per}.}

\section{Results}

The following theorem is the high ground for the problem.
\begin{thm}
\label{thm:BA}Let $\Omega$ be a bounded finitely connected domain.
Then $f(z)$ is the projection of $\overline{z}$ onto $A^{2}(\Omega)$
if and only if $\vert z\vert^{2}=F(z)+\overline{F(z)}$ on $\Gamma=\partial\Omega$,
where $F'(z)=f(z)$.
\end{thm}
(Although $F$ can, in a multiply connected domain, be multivalued,
$Re(F)$ can be assumed to be single valued as a solution to the Dirichlet
problem with data $\left|z\right|^{2}$ on $\Gamma$.)
\begin{proof}
First suppose that $\overline{z}-f(z)$ is orthogonal to $A^{2}(\Omega)$
in $L^{2}(\Omega)$. Then for every $z\in\hat{\mathbb{C}}\backslash\overline{\Omega}$
we have that
\[
\int_{\Omega}(\overline{\zeta}-f(\zeta))\frac{1}{\overline{\zeta-z}}dA(\zeta)=0=\int_{\Omega}(\zeta-\overline{f(\zeta)})\frac{1}{\zeta-z}dA(\zeta).
\]
Then, by Green's Theorem, for any single valued branch of $F$, where
$F'=f$, we have that
\[
\int_{\Gamma}(\vert\zeta\vert^{2}-\overline{F(\zeta)})\frac{1}{\zeta-z}d\zeta=0.
\]
By the F. and M. Riesz Theorem, this happens if and only if we have
\[
\vert\zeta\vert^{2}-\overline{F(\zeta)}=h(\zeta)
\]
on $\Gamma$, where $h(\zeta)$ is analytic in $\Omega$.

Now, since $\vert\zeta\vert^{2}$ is real and we have that $\vert\zeta\vert^{2}=\overline{F(\zeta)}+h(\zeta)$
on $\Gamma$, then it must be that 
\[
\overline{F(\zeta)}+h(\zeta)=F(\zeta)+\overline{h(\zeta)},
\]
which implies that $h=F$.

Conversely, if $\vert\zeta\vert^{2}-\overline{F(\zeta)}=h(\zeta)$
on $\Gamma$ for some $h(\zeta)$ analytic in $\Omega$, then we have
that for all $z\in\hat{\mathbb{C}}\backslash\overline{\Omega}$,
\[
0=\int_{\Gamma}(\vert\zeta\vert^{2}-\overline{F(\zeta)})\frac{1}{\zeta-z}d\zeta
\]
\[
=\int_{\Omega}(\zeta-\overline{F'(\zeta)})\frac{1}{\zeta-z}dA(\zeta),
\]
and so we have that $\overline{\zeta}-F'(\zeta)$ is orthogonal to
$A^{2}(\Omega)$.

.
\end{proof}
This argument is similar to that of Khavinson and Stylianopoulos in
\cite{KhavStyl}. The following is an immediate corollary.
\begin{cor}
The best approximation to $\overline{z}$ in $A^{2}(\Omega)$ is a
polynomial if and only if the Dirichlet problem with data $\left|z\right|^{2}$
has a real-valued polynomial solution. Similarly, the best approximation
to $\overline{z}$ in $A^{2}(\Omega)$ is a rational function if and
only if the Dirichlet problem with data $\left|z\right|^{2}$ has
a solution which is the sum of a rational function and a finite linear
combination of logarithmic potentials of real point charges located
in the complement of $\Omega$.
\end{cor}
The following theorem, loosely speaking, shows that increasing the
connectivity of the domain essentially improves the approximation.
\begin{thm}
Let $\Omega$ be a finitely connected domain and let $f(z)$ be the
best approximation to $\overline{z}$ in $A^{2}(\Omega)$. Then $f$
must have at least one singularity in every bounded component of the
complement.\end{thm}
\begin{proof}
Suppose $\partial\Omega=\Gamma=\cup_{i=1}^{n}\Gamma_{i}$ where $\Gamma_{i}$
is a Jordan curve for each $i$. By Theorem \ref{thm:BA}, we must
have that $\vert z\vert^{2}-2ReF=0$ on $\Gamma$ where $F'=f$. Suppose
that there is a bounded component $K$ of the complement of $\Omega$
such that $f$ is analytic in $G:=\Omega\cup K$. Without loss of
generality we will assume $\partial G=\cup_{i=1}^{n-1}\Gamma_{i}$.
Then $\vert z\vert^{2}-2ReF$ is subharmonic in $G$ and vanishes
on $\partial G$. However since $\vert z\vert^{2}-2ReF$ cannot be
constant in $G,$ it must be that $\vert z\vert^{2}-2ReF<0$ in $G$.
In particular it cannot vanish on $\Gamma_{n}$. 
\end{proof}
The following noteworthy corollary is now immediate.
\begin{cor}
\label{cor:PolySimpConn}If $\Omega$ is a finitely connected domain,
and the best approximation to $\overline{z}$ is a polynomial, then
$\Omega$ must be simply connected and $\partial\Omega$ is algebraic.
\end{cor}
The converse to Corollary \ref{cor:PolySimpConn} is false. In Section
3, we will give an example of a simply connected domain where the
best approximation to $\overline{z}$ is a rational function. Corollary
\ref{cor:PolySimpConn} implies that if the best approximation to
$\overline{z}$ is a polynomial then the boundary of $\Omega$, $\partial\Omega$,
possesses the Schwarz function (cf. \cite{Shapiro}). There is a connection
between the best approximation to $\overline{z}$ in $A^{2}(\Omega)$
and the Schwarz function of $\partial\Omega$. We record this connection
in the following proposition.
\begin{prop}
If $\Omega$ is a simply connected domain, and if the best approximation
to $\overline{z}$ is a polynomial of degree at least 1, then the
Schwarz function of $\Gamma=\partial\Omega$ cannot be meromorphic
in $\Omega$. Further, when the best approximation is a polynomial
the Schwarz function of the corresponding domain must have algebraic
singularities and no finite poles unless $\Omega$ is a disk.\end{prop}
\begin{proof}
Suppose that $S(z)$ is the Schwarz function of $\Gamma=\partial\Omega$
and $p(z)$, a polynomial of degree $n-1$, is the best approximation
to $\overline{z}$ in $A^{2}(\Omega)$ with anti-derivative $P(z)$.
By Theorem \ref{thm:BA}, $zS(z)=P(z)+\overline{P(z)}=P(z)+P^{\#}(S(z))$
on $\Gamma$, where $P^{\#}(z)=\overline{P(\overline{z})}$. If $S$
has a pole of order $k$ at some $z_{0}\neq0$, then $zS(z)$ has
a pole of order $k$ at $z_{0}$ while $P^{\#}(S(z))$ has a pole
of order $nk$ at $z_{0}$. Thus $n\leq1$. If $z_{0}=0$, and $k\geq2$,
then the same argument applies. If $z_{0}=0$ and $k=1$, then $p$
is constant and $\Gamma$ is a circle. Since $S$ is meromorphic in
$\Omega$ if and only if the conformal map $\varphi:\mathbb{D}\rightarrow\Omega$
is a rational function, this shows that if $\Omega$ is a quadrature
domain which is not a disk, then the best approximation to $\overline{z}$
cannot be a polynomial (cf. \cite[pp.17-19]{Shapiro}).
\end{proof}
We now look at some examples illustrating the above results.

\section{Examples}

The following examples were generated using Maple by plotting the
boundary curve $\left|z\right|^{2}-1=Const\:\Re(F(z))$ where, by
Theorem \ref{thm:BA}, $f(z)=\frac{F'(z)}{2}$ is the best approximation
to $\overline{z}$ in $A^{2}(\Omega$), and $\Re(F(z))$ is the real
part of $F(z)$. Since $F$ is unique up to a constant of integration,
all such examples will be similar perturbations of a disk.

Note in the next few examples with best approximation $Cz^{k}$, the
associated domains have the $k+1$ fold symmetry inherited from the
$k$ fold symmetry of the best approximation.

\begin{figure}[H]
\includegraphics[scale=0.5]{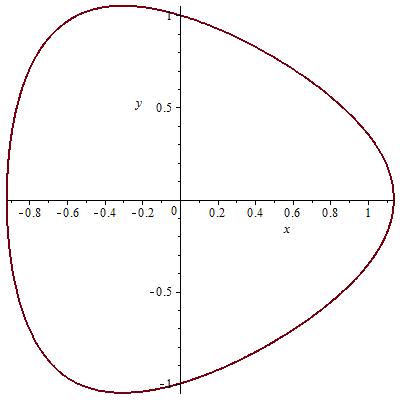}

\caption{}
\end{figure}

In Figure 3.1, the best approximation to $\overline{z}$ is $\frac{3z^{2}}{10}$.

\begin{figure}[H]
\includegraphics[scale=0.5]{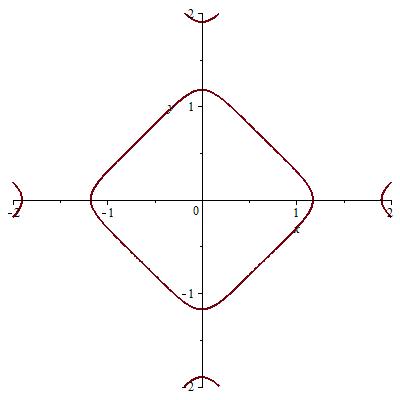}

\caption{}
\end{figure}

In Figure 3.2, the best approximation to $\overline{z}$ is $\frac{2z^{3}}{5}$.

\begin{figure}[H]
\includegraphics[scale=0.5]{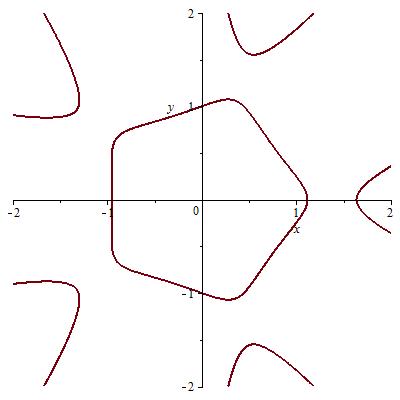}

\caption{}
\end{figure}

In Figure 3.3, the best approximation to $\overline{z}$ is $\frac{5z^{4}}{14}$.

The following example shows that the best approximation may be a rational
function even when the domain is simply connected. Thus while Corollary
\ref{cor:PolySimpConn} guarantees that $\Omega$ is simply connected
whenever the best approximation to $\overline{z}$ is an entire function,
the converse is not true.

\begin{figure}[H]
\includegraphics[scale=0.5]{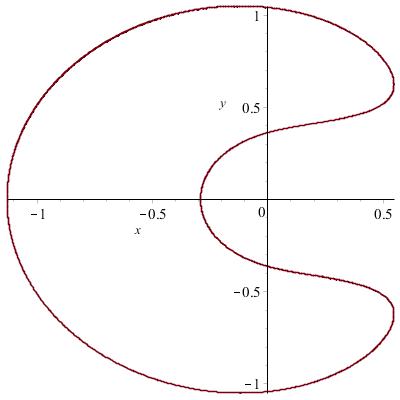}

\caption{}
\end{figure}

In this domain, the best approximation to $\overline{z}$ is $f(z)=\frac{1}{3z}+\frac{1}{5(z-\frac{1}{2})}.$

The constant(s) involved also play a strong role in the shape, and
even connectivity of the domain, as the following pictures shows.

\begin{figure}[H]
\includegraphics[scale=0.5]{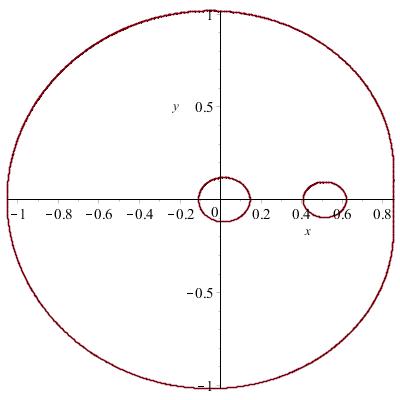}

\caption{}
\end{figure}

In Figure 3.5, the best approximation to $\overline{z}$ is $f(z)=\frac{1}{7z}+\frac{1}{10(z-\frac{1}{2})}.$

\begin{figure}[H]
\includegraphics[scale=0.5]{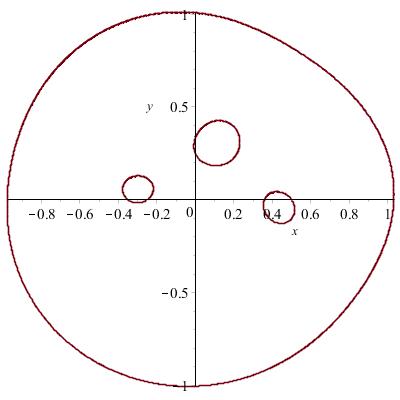}

\caption{}
\end{figure}

In Figure 3.6, the best approximation to $\overline{z}$ is $f(z)=-\frac{3z^{2}-2(\frac{1}{4}-\frac{1}{3}i)z-\frac{1}{8}+\frac{1}{12}i}{40(z-\frac{1}{2})^{2}(z-\frac{i}{3})^{2}(z+\frac{1}{4})^{2}}$.

\begin{figure}[H]
\includegraphics[scale=0.5]{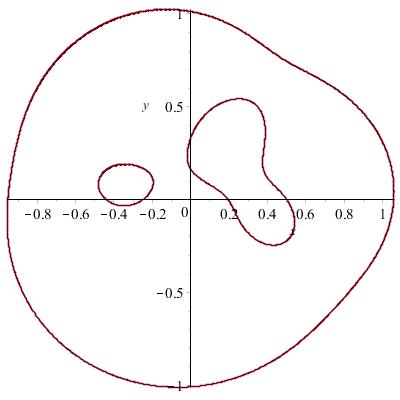}

\caption{}
\end{figure}

In Figure 3.7, the best approximation to $\overline{z}$ is $f(z)=-\frac{3z^{2}-2(\frac{1}{4}-\frac{1}{3}i)z-\frac{1}{8}+\frac{1}{12}i}{10(z-\frac{1}{2})^{2}(z-\frac{i}{3})^{2}(z+\frac{1}{4})^{2}}$.

\begin{figure}[H]
\includegraphics[scale=0.5]{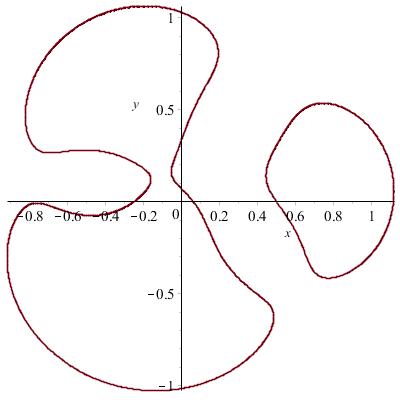}

\caption{}
\end{figure}

In Figure 3.8, the best approximation to $\overline{z}$ is $f(z)=-\frac{3z^{2}-2(\frac{1}{4}-\frac{1}{3}i)z-\frac{1}{8}+\frac{1}{12}i}{8(z-\frac{1}{2})^{2}(z-\frac{i}{3})^{2}(z+\frac{1}{4})^{2}}$.
(It should be noted that in all of the above examples, the poles lie
outside of $\overline{\Omega}$.)

As the order of the pole of the best approximation increases we see
$k-1$ symmetric loops separating the pole from the domain. (Here
$k$ is the order of the pole of the best approximation) .

\begin{figure}[H]
\includegraphics[scale=0.5]{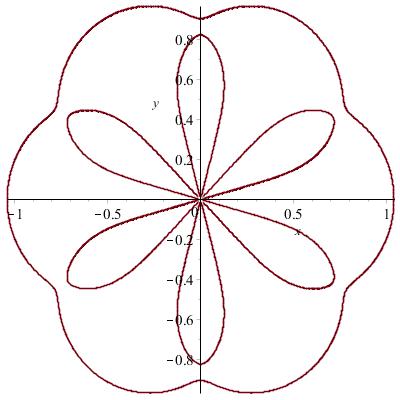}

\caption{}
\end{figure}

In Figure 3.9 the best approximation to $\overline{z}$ is $f(z)=\frac{-3}{10z^{7}}$.
(It should be noted that the loops do not pass through 0. So 0 does
not belong to $\overline{\Omega}$!)

\section{Bergman Analytic Content}

In \cite{GuadKhav} the authors expanded the notion of analytic content,
$\lambda(\Omega):=\inf_{f\in H^{\infty}(\Omega)}\left\Vert \overline{z}-f\right\Vert _{\infty}$
defined in \cite{BeneteauKhavinson} and \cite{Khav}, to Bergman
and Smirnov spaces context. The following ``isoperimetric sandwich''
goes back to \cite{Khav}:
\[
\frac{2A(\Omega)}{Per(\Omega)}\leq\lambda(\Omega)\le\sqrt{\frac{A(\Omega)}{\pi}},
\]
where $A(\Omega)$ is the area of $\Omega$, and $Per(\Omega)$ is
the perimeter of its boundary. Here the upper bound is due to Alexander
(cf. \cite{Alexander}), and the lower bound is due to D. Khavinson
(cf. \cite{BeneteauKhavinson} , \cite{GamelinKhavin}, and\cite{Khav}). 

Following \cite{GuadKhav}, we define $\lambda_{A^{2}}(\Omega):=\inf_{f\in A^{2}(\Omega)}\left\Vert \overline{z}-f\right\Vert _{2}$.
\begin{thm}
\label{thm:Iso-Per}If $\Omega$ is a simply connected domain with
a piecewise smooth boundary, then 
\[
\sqrt{\rho(\Omega)}\leq\lambda_{A^{2}}(\Omega)\leq\frac{Area(\Omega)}{\sqrt{2\pi}},
\]
where $\rho(\Omega)$ is the torsional rigidity of $\Omega$ (cf.
\cite[pg. 24]{PolyaSzego}). \end{thm}
\begin{proof}
To see the lower bound, we note that by duality
\begin{equation}
\lambda_{A^{2}}(\Omega):=\inf_{f\in A^{2}(\Omega)}\left\Vert \overline{z}-f\right\Vert _{2}=\sup_{g\in(A^{2}(\Omega))^{\perp}}\left|\frac{1}{\left\Vert g\right\Vert _{2}}\int_{\Omega}\overline{zg}dA(z)\right|.\label{eq:4.1}
\end{equation}
By Khavin's lemma (cf. \cite{BellFergLund} ,\cite{GuadKhav} and
\cite{Shapiro}), we have that
\[
(A^{2}(\Omega))^{\perp}:=\{\frac{\partial u}{\partial z}|\: u\in W_{0}^{1,2}(\Omega)\},
\]
where $W_{0}^{1,2}(\Omega)$ is the standard Sobolev space of functions
with square-integrable gradients and vanishing boundary values. Thus,
integrating by parts, (\ref{eq:4.1}) can be written as 
\[
\lambda_{A^{2}(\Omega)}=\sup_{u\in W_{0}^{1,2}(\Omega)}\frac{1}{\left\Vert \frac{\partial u}{\partial z}\right\Vert _{2}}\left|\int_{\Omega}udA(z)\right|.
\]
Any particular choice of $u(z)$ will thus yield a lower bound. Suppose
we choose $u(z)$ to be the stress function satisfying
\[
\begin{cases}
\Delta u=-2\\
u|_{\partial\Omega}=0
\end{cases}
\]
 (cf. \cite{BellFergLund} and \cite{PolyaSzego}). Then, since $u(z)$
is real-valued, we have that $\left\Vert \frac{\partial u}{\partial z}\right\Vert _{2}=\frac{1}{2}\left\Vert \nabla u\right\Vert _{2}$
and 
\[
\frac{1}{\left\Vert \frac{\partial u}{\partial z}\right\Vert _{2}}\left|\int_{\Omega}udA(z)\right|=\frac{2\left|\int_{\Omega}udA(z)\right|}{\left\Vert \nabla u\right\Vert _{L^{2}(\Omega)}}=\sqrt{\rho(\Omega)},
\]
(cf. \cite{BellFergLund} and \cite{OlsenReg}). Thus, 
\begin{equation}
\lambda_{A^{2}}(\Omega)\geq\sqrt{\rho(\Omega)}.\label{eq:4.2}
\end{equation}

To prove the upper bound, observe that 
\[
\lambda_{A^{2}}^{2}(\Omega)=\left\Vert \overline{z}\right\Vert ^{2}-\left\Vert P(\overline{z})\right\Vert ^{2},
\]
where $P$ is the Bergman projection. Let $T_{z}$ be the Toeplitz
operator acting on $A^{2}(\Omega)$ with symbol $\varphi(z)=z$, and
let $[T_{z}^{*},T_{z}]=T_{z}^{*}T_{z}-T_{z}T_{z}^{*}$ be the self-commutator
of $T_{z}$. In \cite{OlsenReg}, it was proved that
\[
\left\Vert [T_{z}^{*},T_{z}]\right\Vert =\sup_{g\in A_{1}^{2}(\Omega)}(\left\Vert \overline{z}g\right\Vert ^{2}-\left\Vert P(\overline{z}g)\right\Vert ^{2})\leq\frac{Area(\Omega)}{2\pi},
\]
where $A_{1}^{2}(\Omega)=\{g\in A^{2}(\Omega):\;\left\Vert g\right\Vert _{2}=1\}$.
Taking $g=\frac{1}{\sqrt{Area(\Omega)}}$ yields 
\[
\frac{1}{Area(\Omega)}(\left\Vert \overline{z}\right\Vert ^{2}-\left\Vert P(\overline{z})\right\Vert ^{2})\leq\frac{Area(\Omega)}{2\pi},
\]
and the upper bound follows. 
\end{proof}
The celebrated St. Venant inequality (cf. \cite{PolyaSzego}) follows
immediately.
\begin{cor}
Let $\Omega$ be a simply connected domain. Then 
\[
\rho(\Omega)\leq\frac{Area^{2}(\Omega)}{2\pi}.
\]

\end{cor}

\section{Concluding Remarks}

Recall that for all $u\in W_{0}^{1,2}(\Omega)$, we may write
\begin{equation}
\left|\int_{\Omega}u(z)dA(z)\right|=\left|\int_{\Omega}\frac{-1}{\pi}\int_{\Omega}\frac{\partial u}{\partial\overline{\zeta}}\frac{1}{\zeta-z}dA(\zeta)dA(z)\right|.\label{eq:5.1}
\end{equation}
Applying Fubini's Theorem and the Cauchy-Schwartz inequality, we find
that
\begin{equation}
\left|\int_{\Omega}u(z)dA(z)\right|\leq\left\Vert \frac{\partial u}{\partial\overline{z}}\right\Vert _{2}\left\Vert \frac{1}{\pi}\int_{\Omega}\frac{dA(z)}{z-\zeta}\right\Vert _{2}.\label{eq:5.2}
\end{equation}
In \cite{Dostanic1} and \cite{Dostanic2}, (also cf. \cite{AndersonKhavinsonLomon})
it was proved that the Cauchy integral operator $C:L^{2}(\Omega)\rightarrow L^{2}(\Omega)$,
defined by 
\[
Cf(z)=\frac{-1}{\pi}\int_{\Omega}\frac{f(\zeta)}{\zeta-z}dA(\zeta),
\]
 has norm $\frac{2}{\sqrt{\Lambda_{1}}}$ whenever $\Omega$ is a
simply connected domain with a piecewise smooth boundary, and $\Lambda_{1}$
is the smallest positive eigenvalue of the Dirichlet Laplacian,
\[
\begin{cases}
-\Delta u=\Lambda u\\
u|_{\partial\Omega}=0
\end{cases}.
\]
Further, by the Faber-Krahn inequality , cf. \cite[pp. 18, 98]{PolyaSzego}
and \cite[p. 104]{Bandle}, we have that 
\[
\frac{2}{\sqrt{\Lambda_{1}}}\leq\frac{2}{j_{0}}\sqrt{\frac{Area(\Omega)}{\pi}},
\]
where $j_{0}$ is the smallest positive zero of the Bessel function
$J_{0}(x)=\sum_{k=0}^{\infty}\frac{(-1)^{k}}{(k!)^{2}}(\frac{x}{k})^{2k}$.
Combining the above inequality with (\ref{eq:5.2}) we obtain 
\begin{equation}
\frac{1}{\left\Vert \frac{\partial u}{\partial\overline{z}}\right\Vert _{2}}\left|\int_{\Omega}udA(z)\right|\leq\frac{2}{j_{0}}\frac{Area(\Omega)}{\sqrt{\pi}}.\label{eq:5.3}
\end{equation}
This together with (\ref{eq:4.2}) and (\ref{eq:5.2}), yields an
isoperimetric inequality:
\[
\rho(\Omega)\leq\frac{4Area^{2}(\Omega)}{j_{0}^{2}\pi}.
\]
 However, this is a coarser upper bound than that found above since
$\frac{2}{j_{0}}\geq\frac{1}{\sqrt{2}}$. Since this upper bound depends
entirely on $\left\Vert \frac{1}{\pi}\int_{\Omega}\frac{dA(z)}{z-\zeta}\right\Vert _{2}$,
and since in the case when $\Omega$ is a disk $D$ we find that $\left\Vert \frac{1}{\pi}\int_{D}\frac{dA(z)}{z-\zeta}\right\Vert _{2}=\frac{Area(D)}{\sqrt{2\pi}}$,
we conjecture, in the spirit of the Ahlfors-Beurling inequality (cf.
\cite{AhlforsBeurling} and \cite{GamelinKhavin}), that 
\[
\left\Vert \frac{1}{\pi}\int_{\Omega}\frac{dA(z)}{z-\zeta}\right\Vert _{2}\leq\frac{Area(\Omega)}{\sqrt{2\pi}}.
\]
If true, this would provide an alternate proof to the upper bound
for Bergman analytic content, as well as a more direct proof of the
St. Venant inequality.

One is tempted to ask if any connection can be made between ``nice''
best approximations and the order of algebraic singularities of the
Schwarz function. For example when $\Omega$ is an ellipse, the Schwarz
function has square root singularities at the foci, and the best approximation
to $\overline{z}$ is a linear function. 

We would also like to find bounds on constants $C$ which guarantee
that the solution to the equation $\vert z\vert^{2}-1=C(z^{n}+\overline{z}^{n})$
is a curve which bounds a Jordan domain. This seems to depend on $n$. 

It would also be interesting to examine similar questions for the
Bergman space $A^{p}(\Omega)$ when $p\neq2$, as well as similar
questions for the best approximation of $\left|z\right|^{2}$ in $L_{h}^{2}(\Omega)$,
the closed subspace of functions harmonic in $\Omega$ and square
integrable with respect to area. However, it's not clear what the
analog of Theorem \ref{thm:BA} would be in this case. Mimicking the
proof of Theorem \ref{thm:BA} runs aground quickly.

\end{document}